\documentclass[12pt,reqno]{amsart}

\usepackage{latexsym, amsmath, amscd, amssymb, amsthm}
\usepackage{bm, mathrsfs, euscript}
\usepackage{color}
\usepackage[alphabetic]{amsrefs}
\usepackage[all]{xy}
\usepackage{verbatim}
\usepackage{hyperref}
\usepackage{microtype}

\usepackage[centering, includeheadfoot, hmargin=1in, vmargin=1in,
  headheight=30.4pt]{geometry}

\newtheorem{lemma}{Lemma}[section]
\newtheorem{theorem}[lemma]{Theorem}

\newtheorem{proposition}[lemma]{Proposition}
\theoremstyle{definition}

\newtheorem{example}[lemma]{Example}
\newtheorem{remark}[lemma]{Remark}

\theoremstyle{remark}

\newenvironment{smallarray}[1]
 {\null\,\vcenter\bgroup\scriptsize
  \renewcommand{\arraystretch}{0.7}%
  \arraycolsep=.13885em
  \hbox\bgroup$\array{@{}#1@{}}}{\endarray$
  \egroup\egroup\,\null}

\newcommand{\M}{\overline{M}_{0,6}}
\newcommand{\git}{\ensuremath{\operatorname{/\!\!/}}}
\newcommand{\MOn}{\overline{M}_{0,n}}
\newcommand{\C}{\mathbb C}

\newcommand{\Q}{\mathbb Q}

\DeclareMathOperator{\PGL}{PGL}

\DeclareMathOperator{\Hom}{Hom}

\DeclareMathOperator{\Spec}{Spec}

\DeclareMathOperator{\trop}{trop}

\DeclareMathOperator{\Eff}{Eff}
\DeclareMathOperator{\Cox}{Cox}
\DeclareMathOperator{\Pic}{Pic}

\DeclareMathOperator{\pos}{pos}
\DeclareMathOperator{\spann}{span}
\DeclareMathOperator{\ord}{ord}

\DeclareMathOperator{\divv}{div}
\DeclareMathOperator{\Sym}{Sym}
\DeclareMathOperator{\Mov}{Mov}

\begin{document}

\title{A presentation for the Cox ring of $\overline{M}_{0,6}$}

\author{Martha Bernal Guill\'en}

\address{Unidad Acad\'emica de Matem\'aticas\\ Universidad Aut\'onoma de Zacatecas\\
C.P. 98000  \\ Zacatecas, Zacatecas, M\'exico}
\email{m.m.bernal.guillen@gmail.com}

\author{Diane Maclagan}
\address{Mathematics Institute\\ University of Warwick\\
Coventry, CV4 7AL\\ United Kingdom}
\email{D.Maclagan@warwick.ac.uk}

\begin{abstract}
We compute the relations in the Cox ring of the moduli space $\M$.
This gives a presentation of the Cox ring as a
quotient of a polynomial ring with 40 generators by an ideal with 225
generators that come in 5 symmetry classes.
\end{abstract}

\maketitle 

\section{Introduction}

Mori dream spaces, introduced by Hu and Keel in \cite{HuKeelMDS}, are
varieties for which the minimal model program works particularly well.
In particular, the nef cone is polyhedral and generated by finitely
many semi-ample divisors, and the movable cone is the union of finitely
many polyhedral cones which are the nef cones of small $\mathbb Q$-factorial
modifications of the original variety.  While this is a strong
condition, one consequence of \cite{BCHM} is that log Fano varieties
are Mori dream spaces.

In the aftermath of \cite{BCHM} understanding the birational geometry
of the moduli space of curves, following the Hassett-Keel program, has
become a major industry; see, for example, \cite{Hassett,
  HassettHyeon, Alperetal}.  A natural question is whether
$\overline{M}_{g,n}$ is a Mori dream space; the case $g=0$ was raised
in the original paper~\cite{HuKeelMDS}.  The answer seems to be
increasingly no; $\overline{M}_{g,n}$ is not a Mori dream space for
large $g$ and $n$ \cite{ChenCoskun,Mullane}, and $\overline{M}_{0,n}$
is not a Mori dream space for $n \geq 10$
\cite{CastravetTevelevNotMDS, GonzalezKaru, HausenKeicherLaface}.

For smaller $n$, however, the situation is more tractable.
In \cite{Castravet} Castravet showed that the moduli space $\M$ is a
Mori dream space; since $\M$ is log Fano, this now also follows from \cite{BCHM}.
Castravet's work builds on earlier work of Keel, Vermeire~\cite{Vermeire}, and
  Hassett and Tschinkel~\cite{HassettTschinkel} who showed that the
  effective cone of $\M$ is generated by the $25$ boundary divisors,
  and by $15$ extra divisors.  She showed that the Cox ring of
  $\M$ is generated by sections of these $40$ divisors.  In this paper
  we compute the relations between these generators to give a
  presentation for the Cox ring.

The $25$ boundary divisors are indexed by subsets $I \subset
\{1,2,3,4,5,6\}$ with $2 \leq |I| \leq 3$ and $1 \in I$ if $|I|=3$.
We denote this set by $\mathcal I$, and a section of the divisor
$\delta_I$ indexed by $I$ by $x_I$.  The $15$ extra Keel-Vermeire
divisors are indexed by permutations in the symmetric group $S_6$ of
the form $(ij)(kl)(mn)$.  We denote the set of these permutations
by $\Pi$, and for $\pi \in \Pi$ write $y_{\pi}$ for a section of the
corresponding divisor.
The main result of this paper is:

\begin{theorem} \label{t:maintheorem}
Let $S = \mathbb C[x_I, y_{\pi} : I \in \mathcal I, \pi \in \Pi]$.  The Cox ring
of $\M$ is $S/I_{\M}$, where $I_{\M}$ is an ideal with 225 generators,
which come in five symmetry classes:

\begin{enumerate}
\item $15$ of the form:
\begin{equation*}
x_{ij}x_{kl}x_{ijn}x_{kln}-x_{ik}x_{jl}x_{ikn}x_{jln}+x_{il}x_{jk}x_{iln}x_{jkn}
\end{equation*}
for $1\leq i<j<k<l\leq 6$, $m<n$ and $\{i,j,k,l,m,n\}=\{1,2,3,4,5,6\}$;

\item  $60$ of the form:
\begin{equation*}
x_{1ik}y_{(1m)(ij)(kl)}+x_{1j}x_{mk}x_{il}x_{1mj}x_{1mk}x_{1ij}+ x_{1l}x_{mi}x_{jk}x_{1mi}x_{1ml}x_{1kl},
\end{equation*}
for $\{i,j,k,l,m\}=\{2,3,4,5,6\}$;

\item  $45$ of the form:
\begin{equation*}
x_{ij}y_{(ij)(kl)(mn)} +  x_{ik}x_{il}x_{jm}x_{jn}x_{ikl}^2 -  x_{im}x_{in}x_{jl}x_{jk}x_{imn}^2,
\end{equation*}
for $\{i,j,k,l,m,n\}=\{1,2,3,4,5,6\}$ and $i<j$;

\item  $45$ of the form:
\begin{equation*}
  y_{(ij)(kl)(mn)}y_{(ij)(km)(ln)}-x_{il}x_{lj}x_{jm}x_{mi}x_{nk}^{2}x_{ijm}^{2}x_{ijl}^{2} +x_{in}x_{nj}x_{jk}x_{ki}x_{ml}^{2}x_{ijk}^{2}x_{ijn}^{2},
\end{equation*}
for $\{i,j,k,l,m,n\} = \{1,2,3,4,5,6\}$;

\item  $60$ of the form:
\begin{multline*}
\qquad \quad y_{(1i)(jk)(lm)}y_{(1j)(il)(km)} - x_{1k}x_{1l}x_{ik}x_{im}x_{jl}x_{jm}x_{1ik}x_{1il}x_{1jk}x_{1jl}\\
+x_{1k}x_{1l}x_{ij}x_{im}x_{jm}x_{kl}x_{1ij}x_{1il}x_{1jk}x_{1kl} -x_{1k}x_{1m}x_{ij}x_{im}x_{jl}x_{kl}x_{1ij}x_{1im}x_{1jk}x_{1km}\\
-x_{1l}x_{1m}x_{ij}x_{ik} x_{jm} x_{kl} x_{1ij} x_{1il} x_{1jm} x_{1lm},
\end{multline*}
for $\{i,j,k,l,m\} = \{2,3,4,5,6\}$.
\end{enumerate}
Here we follow the convention that $x_{ij}\!=\!-x_{ji}$,
$x_{ijk}\!=\!-x_{jik}=\!-x_{ikj}$, $x_{ijk}\!=\!x_{lmn}$ for
\small$\{i,j,k,l,m,n \}= \{1,2,3,4,5,6\}$\normalsize, with $i<j<k$ and
$l<m<n$, and $y_{(ij)(kl)(mn)}=-y_{(kl)(ij)(mn)}=-y_{(ij)(mn)(kl)}$.
\end{theorem}

The proof exploits the embedding of $\M$ into a $24$-dimensional toric
variety determined by the generators of the Cox ring.  It can be done
with various degrees of computer assistance.  One contribution of this
paper is that we have taken care when choices were necessary so that
the $S_6$ action is transparent, and the representatives given here
for the generators have a simple form.

One application
of a complete finite presentation of the Cox ring of a variety $X$ is
that it allows the computation of all small birational models of $X$.
Using an earlier circulated version of this paper B\"ohm, Keicher, and
Ren have computed all the chambers of the Mori chamber decomposition
of $\M$ \cite{BKR}, finding $176, 512, 180$ maximal chambers that form $249,
605$ orbits under the $S_6$-action.  While this suggests that the
problem of computing all small birational models is intractable, in
Section~\ref{s:smallbirationalmodels} we consider some simplifications
which imply that the chambers of the movable cone are simpler.
Identifying these models has already been considered in \cite{Moon}
for the two-dimensional slice of the Ner\'on-Severi group containing
$K_{\M}$ and $\sum_{i=1}^6 \psi_i$.  This slice intersects two other
chambers, corresponding to the symmetric GIT quotient of $(\mathbb
P^1)^6$, and the symmetric Veronese quotient.

The structure of the paper is as follows. 
 In Section~\ref{s:background} we develop the necessary preliminaries
 on the generators of $\Cox(\M)$.  In section~\ref{s:coxequations} we
 explain how to compute the relations between these generators.
 Theorem~\ref{t:maintheorem} is proved in Section~\ref{s:maintheorem}.
 Finally, in Section~\ref{s:smallbirationalmodels} we consider the
 problem of using Theorem~\ref{t:maintheorem} to compute the small
 birational models of $\M$, and provide some computational
 simplifications.

\section{The Cox ring of $\M$}
\label{s:background}

In this section we recall some background on the Cox ring of $\M$.  The key
result of this section is Lemma~\ref{l:Qpiequation}, which will be
used to prove Theorem~\ref{t:maintheorem}.

The moduli space $\M$ of stable genus zero curves with $6$ marked
points is a smooth $3$-dimensional variety that compactifies
$M_{0,6} \cong (\mathbb C^* \setminus \{1\})^3 \setminus \bigcup_{1 \leq
  i <j \leq 3}\{x_i =x_j \}$.  The boundary $\M \setminus M_{0,6}$
parameterizes stable genus zero curves with $6$ marked points.  The
boundary is a union of $25$ boundary divisors $\delta_I$, where $I
\subset \{1,\dots, 6\}$ with $|I|=2$ or $|I|=3$, and we identify
$\delta_I$ and $\delta_{\{1,\dots,6\} \setminus I}$ when $|I|=3$.  The
boundary divisor $\delta_I$ is the closure of the locus in $\M$
consisting of those nodal curves with two irreducible components,
where the points with labels in $I$ lie on one component, and the
other points on the other.

The Picard group of $\M$ is isomorphic to $\mathbb Z^{16}$ and is
spanned by the boundary divisors.  
Given a basis $\mathcal B = \{E_1,\dots,E_{16}\}$ for $\Pic(\M)$, the Cox ring, first
introduced for any $\mathbb Q$-factorial variety $X$ with $\Pic(X)
\cong \mathbb Z^m$ in \cite{HuKeelMDS} is
$$\Cox(\M) = \oplus_{a \in \mathbb Z^{16}} H^0(\M, \mathcal 
  O(a_1E_1+\dots+a_{16}E_{16}) ).$$ This is naturally graded by $\Pic(\M)$.

 The pseudo-effective cone
$\Eff(\M)$ is generated by the classes of the $25$ boundary divisors
and the classes of $15$ other effective divisors,
the Keel-Vermeire divisors~\cite{HassettTschinkel}, whose description
we recall below.  In \cite{Castravet} Castravet showed that sections
of these divisors suffice to generate the Cox ring of $\M$.

\begin{theorem}[{\cite[Theorem 1.4]{Castravet}}] \label{t:Castravet}
The Cox ring of $\overline{M}_{0,6}$ is generated by sections of the
boundary divisors $\delta_I$ and the Keel-Vermeire divisors $Q_{\pi}$.
\end{theorem}

In this paper we extend Castravet's result to describe in addition the
relations on this ring.  In the rest of this section we explicitly
describe the complement of these divisors in $\M$.

In \cite{KapranovChow} Kapranov shows that $\MOn$ is the Chow quotient of the
Grassmannian $G(2,n)$ by the action of the torus $(\mathbb C^*)^{n-1}
\cong (\mathbb C^*)^n/\mathbb C^*$ induced from the action of
$(\mathbb C^*)^{n-1}$ on $\mathbb P^{n-1}$.  Under this identification
we have $M_{0,n} = G^0(2,n)/(\mathbb C^*)^{n-1}$, where $G^0(2,n)$ is
the locus in the Grassmannian $G(2,n)$ where all Pl\"ucker coordinates
are nonzero.  This is naturally a subvariety of the quotient torus
\small$(\mathbb C^*)^{{n \choose 2}-1}/(\mathbb C^*)^{n-1}$\normalsize, whose
coordinate ring we denote by \small$\mathbb C[z_{24}^{\pm 1},
  \dots,z_{n-1n}^{\pm 1}]$\normalsize.  We denote by $\mathcal E$ the index set
\small$\{ \{2,4\},\dots,\{5,6\}\}$ \normalsize of the variables in this Laurent
polynomial ring. 
The following lemma is the $n=6$ case of \cite[Proposition
  6.4]{GibneyMaclaganEquations}.
\begin{lemma} \label{l:M06equations}
The moduli space $M_{0,6}$ is an affine variety with coordinate ring 
$$\mathbb C[z_{24}^{\pm 1},z_{25}^{\pm 1},z_{26}^{\pm 1},z_{34}^{\pm 1},z_{35}^{\pm 1},z_{36}^{\pm 1}, z_{45}^{\pm 1},
  z_{46}^{\pm 1},z_{56}^{\pm 1}]/  I_6 ,$$
where 
\begin{align*} I_6 = \langle z_{34}-z_{24}+1, z_{35}-z_{25}+1, 
z_{36}-z_{26}+1, z_{45}-z_{25}+z_{24}, z_{46}-z_{26}+z_{24}, z_{56}-z_{26}+z_{25} \rangle.
\end{align*}
\end{lemma}

This follows from the quotient construction, using the following identification of the variables $z_{ij}$.  We denote by $\{ x_{ij} : 1 \leq i < j \leq n\}$  the
coordinates of $\mathbb P^{{n \choose 2}-1}$.  Then 
\begin{equation} \label{eqtn:zij}
z_{ij} = \begin{cases}
 (x_{ij}x_{12}x_{13})/(x_{1i}x_{1j}x_{23}) & \text{ if } i,j \geq 4, \\
 (x_{2j}x_{13})/(x_{1j}x_{23}) & \text{ if } i =2, \\
(x_{3j}x_{12}) / (x_{1j}x_{23}) & \text{ if } i =3.\\
\end{cases}
\end{equation}
For more details see \cite{GibneyMaclaganEquations}*{Section 6}.

Not all effective divisors on $\M$ lie in the cone spanned by the
boundary divisors, by work of Keel and Vermeire~\cite{Vermeire}.  We
now recall their construction.  The symmetric group $S_6$ 
 acts naturally on $\M$ and $M_{0,6}$ by permuting the points.
Let $\pi=(ij)(kl)(mn) \in S_6$ be the product of three disjoint
transpositions.  The divisor $Q_{\pi}$ is the closure in $\M$ of the
fixed locus of the action of $\pi$ on $M_{0,6}$. 
The divisor $Q_{\pi}$ is effective by construction, but its class in
$\Pic(\M)$ cannot be written as a nonnegative combination
of the classes of the boundary divisors.

The next lemma is the key result of this section, and is needed for the proof of Theorem~\ref{t:maintheorem}.

\begin{lemma} \label{l:Qpiequation}
Let $\pi=(12)(34)(56)$.  The intersection of the divisor $Q_{\pi}$ with $M_{0,6}$ is the subvariety of $(\mathbb C^*)^9$ defined by the ideal
$$I_6 + \langle z_{24}-z_{25}z_{26} \rangle. $$
\end{lemma}

\begin{proof}
A point in $M_{0,6}$ corresponds to a configuration of six distinct
points in $\mathbb P^1$, which we may take to be $\{ \infty, 1, 0,
A,B,C \}$, or more formally $\{ [1:0], [1:1], [0:1], [A:1], [B:1],
[C:1] \}$, where $A,B,C \neq 0,1$.  The transposition $\pi$ acts on
this configuration, taking it to $\{ [1:1], [1:0], [A:1], [0:1],
[C:1], [B:1] \}$.  The automorphism of $\mathbb P^1$ given by the
matrix
$$\left( \begin{array}{rr} 1 & -A \\ 1 & -1 \\ \end{array} \right) \in
\PGL(2)$$ takes this configuration to the configuration $\{[1:0],
    [1:1], [0:1], [A:1], [\frac{C-A}{C-1}:1], [\frac{B-A}{B-1}:1] \}$.
    This point is thus fixed if and only if $\frac{C-A}{C-1} = B$ and
    $\frac{B-A}{B-1}=C$.  This means that $B+C-A-BC =0$.  To translate
    this into Pl\"ucker coordinates, we note that the Pl\"ucker
    coordinates of the matrix
$$\left( \begin{array}{rrrrrr} 1 & 1 & 0 & A& B & C \\ 
0 & 1 & 1 & 1 & 1 & 1\\ \end{array} \right)$$
in the order $z_{24},z_{25},\dots,z_{56}$  are:
$$(- A + 1, - B + 1, - C + 1, -A, -B, -C, A - B, A - C, B - C),
$$
so $B+C-A-BC = z_{24}-z_{25}z_{26}$.  
\end{proof}

Equations for the other Keel-Vermeire divisors can be obtained using
the action of the symmetric group $S_6$ on $\mathbb C[T^9]$, which we
now recall.  The action comes from the description of $T^9$ as
$(\mathbb C^*)^{6 \choose 2}/(\mathbb C^*)^6$, or equivalently from
the formula for $z_{ij}$ given in \eqref{eqtn:zij}.  For example, when
$i,j \geq 4$, we have \small $\textnormal{\small (12)} z_{ij} =
\textnormal{\small (12)} (x_{ij}x_{12}x_{13}/(x_{1i}x_{1j}x_{23})) =
-(x_{ij}x_{12}x_{23})/(x_{2i}x_{2j}x_{13}) =
-z_{ij}/z_{2i}z_{2j}$\normalsize.  This is summarized in Table
\ref{t:S6action1}, where the entry in the row labeled by the adjacent
transposition \small $(i (i+1))$\normalsize \hspace{1pt} and column labeled by
$z_{kl}$ is \small $(i (i+1))z_{kl}$\normalsize .

\hspace{-5mm}
\small
\begin{table}[h] 
{                                  
\renewcommand{\arraystretch}{1.1}
\begin{center}
\begin{tabular}[h]{llllllllll}
 & $z_{24}$ & $z_{25}$ & $z_{26}$ & $z_{34}$ & $z_{35}$ & $z_{36}$ & $z_{45}$ & $z_{46}$ & $z_{56}$ \\

$ (12) $& $ z_{24}^{\textnormal{-}1}$ & $ z_{25}^{\textnormal{-}1}$ & $ z_{26}^{\textnormal{-}1}$ &
  -$  z_{34}/\!z_{24}$ & -$ z_{35}/\!z_{25}$& -$ z_{36}/\!z_{26}$ &
  -$ z_{45}/\!z_{24}z_{25}$ & -$ z_{46}/\!z_{24}z_{26}$ &
  -$ z_{56}/\!z_{25}z_{26}$  \\

$ (23) $& -$z_{34}$ & -$z_{35}$ & -$z_{36}$ & -$z_{24}$ & -$z_{25}$ &
  -$z_{26}$ & -$z_{45}$ & -$z_{46}$ & -$z_{56}$\\

$(34) $& $ z_{24}^{\textnormal{-}1}$ & $ z_{25}/\!z_{24}$ & $ z_{26}/\!z_{24}$ &
    -$ z_{34}/\!z_{24}$ & $ z_{45}/\!z_{24}$ & $ z_{46}/\!z_{24}$ &
    $ z_{35}/\!z_{24}$ & $ z_{36}/\!z_{26}$ & $ z_{56}/\!z_{24}$ \\

$(45) $& $z_{25}$ & $z_{24}$ & $z_{26}$ & $z_{35}$ & $z_{34}$ & $z_{36}$
    & -$z_{45}$ & $z_{56}$ & $z_{46}$ \\

$(56) $& $z_{24}$ & $z_{26}$ & $z_{25}$ & $z_{34}$ & $z_{36}$ & $z_{35}$
                                                                   & $z_{46}$ & $z_{45}$ &-$z_{56}$ \\
 &&&&&&&&&\\
\end{tabular}
\end{center}
\caption{The $S_6$ action on $\C[M_{0,6}]$}
\label{t:S6action1}
}
\end{table}
\normalsize

Applying the action to the equation $z_{24}-z_{25}z_{26}$ for
$Q_{(12)(34)(56)}$, we get the equations in Table~\ref{t:QpiEqns}.

\begin{table}[h!] 
{                                  
\renewcommand{\arraystretch}{1.5}
\begin{center}
\begin{tabular}{llll}
$\scriptstyle (12)(34)(56)$ & $z_{24}-z_{25}z_{26}$ & 
$\scriptstyle (12)(35)(46)$ & $z_{25}-z_{24}z_{26}$\\
$\scriptstyle (12)(36)(45)$ & $z_{26}-z_{24}z_{25}$ & 
$\scriptstyle (13)(24)(56)$ & $-z_{34}-z_{35}z_{36}$\\
$\scriptstyle (13)(25)(46)$ & $-z_{35}-z_{34}z_{36}$ &
$\scriptstyle (13)(26)(45)$ & $-z_{36}-z_{34}z_{35}$\\
$\scriptstyle (14)(23)(56)$ & $z_{24}z_{36}-z_{25}z_{46}$ &
$\scriptstyle (14)(25)(36)$ & $-z_{46}+z_{24}z_{56}$\\
$\scriptstyle (14)(26)(35)$ & $-z_{45}-z_{24}z_{56}$ & 
$\scriptstyle (15)(23)(46)$ & $z_{25}z_{36}-z_{24}z_{56}$\\
$\scriptstyle (15)(24)(36)$ & $-z_{56}+z_{25}z_{46}$ &
$\scriptstyle (15)(26)(34)$ & $z_{45}-z_{25}z_{46}$\\
$\scriptstyle (16)(23)(45)$ & $z_{26}z_{35}+z_{24}z_{56}$ &
$\scriptstyle (16)(24)(35)$ & $z_{56}+z_{26}z_{45}$\\
$\scriptstyle (16)(25)(34)$ & $z_{46}-z_{26}z_{45}$ \\
\end{tabular}
\end{center}
\caption{Equations defining the KV-divisors}
\label{t:QpiEqns}
}
\end{table}

We denote the equation for $Q_\pi$ by $f_\pi$. The polynomial
$f_{\pi}$ depends on several choices.  Firstly, adding any element of
the ideal $I_6$ gives the same element of $\C[M_{0,6}]$.  More seriously,
there is a dependence on the choice of $\sigma \in S_6$ with $\pi =
\sigma \cdot \textnormal{\small (12)(34)(56)}$.  For example, $\sigma
=\textnormal{\small (12)}$ gives $\pi = \textnormal{\small
  (12)(34)(56)}$, but $\textnormal{\small (12)} (z_{24}- z_{25}z_{26})
= z_{24}^{-1} - z_{25}^{-1} z_{26}^{-1} =
-f_{(12)(34)(56)}/z_{24}z_{25}z_{26}$.  This is a different element of
$\C[M_{0,6}]$.  The corresponding divisor is still $Q_{(12)(34)(56)}$,
however.  
For the explicit approach taken
here making a choice is necessary at this stage.  The final result
given in Theorem~\ref{t:maintheorem} does not depend on this choice,
however.  

\begin{remark}
We note that $Q_{\pi}$ is not the fixed locus of $\pi$ on all of $\M$.
This fixed locus contains the zero-stratum of $\M$ corresponding to
the curve with four components, three of which contain the pairs of
points $\{i,j\}$, $\{k,l\}$, and $\{m,n\}$, which is not in $Q_{\pi}$.
This can be seen, for example, by noting that $\trop(V(I_6+ \langle
f_{\pi} \rangle))$ does not intersect the cone of $\trop(M_{0,n})$
corresponding to this stratum. 
\end{remark}

The set $Y =M_{0,6}\setminus \bigcup_{\pi\in S_6}(M_{0,6}\cap Q_\pi)$
is an open subvariety of $M_{0,6}$. An easy but important consequence
of Lemma~\ref{l:Qpiequation} is the following.  Recall that $\mathcal
E =\{ \{2,4\},\dots, \{5,6\}\}$, and $\Pi$ is the set of products of
three disjoint transpositions in $S_6$.

\begin{lemma}\label{l:idealY}
The open set $Y=M_{0,6}\setminus\big(\bigcup_\pi {\mathbf
  V}(f_\pi)\big)$ is an affine variety with coordinate ring
$\C[z_{ij}^{\pm 1},u_\pi^{\pm 1}: ij \in \mathcal E, \pi \in
  \Pi]/J_{KV}$ where
\[
  J_{KV}=I_6+\langle  u_{\pi}-f_{\pi} \colon \pi \in \Pi \rangle
 \subset \C[(\mathbb C^*)^{24}].
\]
\end{lemma}
\begin{proof}
For an affine variety $X$ with coordinate ring $K[X]$ and an
irreducible divisor $V(g) \subset X$ the complement $X \setminus V(g)$
has coordinate ring $K[X][u^{\pm 1}]/\langle u-f \rangle$.  The lemma
follows from applying this inductively to the $f_{\pi}$.
\end{proof}

The $S_6$-action on the variables $z_{ij}$ induces an action on the
variables $u_{\pi}$ as well, by imposing the condition $\sigma \cdot
(u_{\pi}-f_{\pi}) = m(u_{\sigma \pi \sigma^{-1}}-f_{\sigma \pi
  \sigma^{-1}})$ for some term $m$.  For example, when $\pi =
(12)(34)(56)$ and $\sigma = (12)$, we have $(12) f_{\pi} =
-1/(z_{24}z_{25}z_{26}) f_{\pi}$, so $(12) u_{\pi} =
-1/(z_{24}z_{25}z_{26}) u_{\pi}$.

\section{Equations for Cox rings}
\label{s:coxequations}

In this section we outline an approach to computing relations for the
Cox ring of a Mori dream space when a set of generators is already
known, and begin to apply this for the case of $\M$.

We first recall notation for Cox rings.  Throughout this section $X$
is a normal $\mathbb Q$-factorial projective variety defined over a field $K$  with
$\Pic(X)_{\mathbb Q} \cong \mathbb Z^r$.
Fix divisors $E_1,\dots,E_r$ whose classes form a basis for
$\Pic(X)_{\mathbb Q}$.  The Cox ring of $X$ is
$$\Cox(X) = \oplus_{\mathbf a \in \mathbb Z^r} H^0(X,\mathcal
O(a_1E_1+\dots + a_r E_r).$$

 The variety $X$ is a {\em Mori dream space}~\cite{HuKeelMDS} if
 $\Cox(X)$ is finitely generated, so we can write $\Cox(X) = S/I$,
 where $S = K[x_1,\dots,x_s]$ for sections $x_i \in
 H^0(X,\mathcal O(D_i))$ of divisors $D_1,\dots,D_S$.  This
 description does involve choosing sections $x_i$, but in many cases,
 including $\M$, the $x_i$ are determined up to scalar multiple.  The
 grading $\deg(x_i)=[D_i] \in \Pic(X)$ induces an action of the torus
 $H = \Hom(\Pic(X),K^*)$ on $\mathbb A^s$ by $\phi \cdot x_i =
 \phi([D_i]) x_i$, and this restricts to an action on $V(I)$.  Fix an
 ample class $\alpha \in \Pic(X)$, and let $X_{\Sigma}$ be the toric
 variety $\mathbb A^{s} \git_{\alpha} H$.  This has torus $T=
 (K^*)^s/H$.  We assume that $\alpha$ is chosen sufficiently
 generically so that the fan $\Sigma$ is simplicial.  This is possible
 as nonsimplicial $\Sigma$ only arise from $\alpha$ on a union
 of hyperplanes in $\Pic(X) \otimes \mathbb R$.  Then $X = V(I)
 \git_{\alpha} H$ embeds into $X_{\Sigma}$. Moreover, the restriction
 $\Pic_\Q(X_\Sigma) \rightarrow \Pic_\Q(X)$ is an isomorphism and it
 induces an isomorphism of the pseudo-effective cones
 $\Eff(X_\Sigma)\rightarrow \Eff(X)$ \cite[Proposition
   2.11]{HuKeelMDS}.

\begin{example} \label{e:M06embedding}
When $X=\M$, the Cox ring has $40$ generators as we now recall.  Let
$\mathcal I = \{ I \subseteq [6] : 2 \leq |I| \leq 3 \text{ and } 1
\in I \text{ if } |I| =3 \}$ be the set of partitions indexing the
boundary divisors in $\M$ and let $\Pi = \{ (ij)(kl)(mn) \} \subseteq
S_6$ be the permutations indexing the Keel-Vermeire divisors.
Consider the polynomial ring $S = \mathbb C[x_I, y_{\pi} : I \in
  \mathcal I, \pi \in \Pi]$ graded by $\deg(x_I)=\delta_{I}$ and
$\deg(y_\pi)=Q_\pi$. Then Theorem~\ref{t:Castravet} implies that there
is a graded surjective algebra homomorphism $S\rightarrow
\Cox(\M)$. Let $I_{\M}$ be its kernel and $H=\Hom(\Pic(\M),\mathbb
C^*) \cong T^{16}$. Then $H$ acts on $\Spec(S)=\mathbb A^{40}$ and
given a sufficiently general ample class $\alpha \in \Pic(\M)$ the GIT
quotient $X_\Sigma=\mathbb A^{40}\git_\alpha H$ is a 24-dimensional
simplicial toric variety with $\M =V(I_{\M})\git_\alpha H \subset
\mathbb A^{40} \git_\alpha H$ \cite[Proposition 2.11]{HuKeelMDS}.  The
intersection of $\M$ with the torus $T^{24}$ of the torus
$X_{\Sigma}:= \mathbb A^{40} \git_\alpha H$ equals $M_{0,6} \setminus
\bigcup_{\pi \in \Pi} Q_{\pi}$.
\end{example}

The next proposition explains how to compute the ideal $I$ of
relations between the generators of the Cox ring.  Recall that when
$X$ is a Mori Dream space the Mori chambers coincide with the GIT chambers
for $X = V(I) \git_{\alpha} H$ \cite[Theorem 2.4]{HuKeelMDS}.  Any
non-moving divisor on the toric variety $X_{\Sigma} =\mathbb A^s
\git_{\alpha} H$ restricts to a non-moving divisor on $X$, so since
the moving cone of $X$ is full-dimensional, the same is true for
$X_{\Sigma}$.  This implies that the polynomial ring $S$ is the Cox
ring of $X_\Sigma$.  One consequence of this
is that the divisors on $X$ defined by generators of $\Cox(X)$ are in
bijection with the torus invariant divisors on $X_{\Sigma}$.

We write $S_{\mathbf{x}}$ for the localization of $S$ at the product
of the variables.  By \cite{CoxToricTotalCoordinate}
$(S_{\mathbf{x}})_{\mathbf{0}}$ is the coordinate ring of the torus
$T$ of $X_{\Sigma}$, where the subscript denotes the degree zero part
in the $\Pic(X)$ grading.

\begin{proposition} \label{p:CoxRingRelations}
Let $X$ be a normal projective Mori Dream Space 
and
let $ X \hookrightarrow X_{\Sigma}$ be a toric embedding given by
choosing $\alpha \in \Pic(X)$ ample.  
Write $E_i$ for the effective divisor
on $X$ corresponding to the $i$th ray of $\Sigma$.
Let $Y= X \cap T$, where $T \cong (K^*)^n$ is the torus of
$X_{\Sigma}$.  Let $I(Y) \subset K[z_1^{\pm 1},\dots,z_n^{\pm 1}]$ be
the ideal of $Y$.  Then $\Cox(X) \cong S/I$, where 
$$I = (\phi(I(Y))S_{\mathbf{x}})\cap S,$$ 
and $\phi \colon K[T]
 \rightarrow S_{\mathbf{x}}$ is the homomorphism 
$\phi \colon K[z_1^{\pm 1},\dots,z_n^{\pm 1}] \rightarrow S_{\mathbf{x}}$ given 
by $$\phi(z_i) = \prod_{j=1}^s x_j^{\ord_{E_j}(z_i)}.$$
\end{proposition}

\begin{proof}
The Cox ring of a normal projective variety is a domain \cite[Proposition 1.4.1.5]{CoxRingBook}.
This means that $I$ is a prime ideal, so $I=IS_{\mathbf{x}} \cap S$.
Note that $IS_{\mathbf{x}}$ is generated by its degree zero part
$(IS_{\mathbf{x}})_0$, so it suffices to show that the map $\phi$
identifies $\phi(I(Y))$ with $(IS_{\mathbf{x}})_0$.  By
\cite[Proposition 2.3]{GibneyMaclaganEquations} it in turn suffices to
show that $\phi$ induces an isomorphism $\Spec(K[z_1^{\pm
    1},\dots,z_n^{\pm 1}]) \cong T \subset X_{\Sigma}$. 

The coordinate function $z_i$ is a unit in $K[Y]$, and so a rational
function on $X$.  Its image $\phi(z_i)$ is also a unit in $K[Y]$.
This means that the divisors $\divv(z_i)$ and $\divv(\phi(z_i))$ on $X$
are supported on the boundary $X \setminus Y = X \setminus T =
\cup_{i=1}^s E_i$.  A rational function $f$ is determined up to scalar
multiplication by its divisor $\divv(f)$.  
 As the ideal $I$ is also defined only up to the multiplication of the
 variables $x_i$ by scalars, it thus suffices to show that
 $\ord_{E_i}(z_j) = \ord_{E_i}(\phi(z_j))$.  This is the case by
 construction.
\end{proof}

To apply Proposition~\ref{p:CoxRingRelations} to $\M$ we need to
compute $\ord_D(z_{ij})$ for $D \in \{ \delta_I, Q_{\pi} : I \in \mathcal I, \pi \in  \Pi\}$ and $\{i,j\}
\in \mathcal E$.  Our method relies on the calculation of the classes
of these divisors $D$ in an explicit basis for $\Pic(\M)$.  As we rely
heavily on the $S_6$-action, we choose a description in which this
action is transparent.  Fix a copy of $\mathbb Z^{16}$ with basis
labelled by $\{\mathbf{e}_i : 1 \leq i \leq 6 \}$ and
$\{\mathbf{e}_{1ij} : 2 \leq i<j \leq 6 \}$.  Set the boundary divisor
$\delta_{1ij}$ equal to $2 \mathbf{e}_{1ij}$, and $\delta_{ij}$ equal
to $\mathbf{e}_i+\mathbf{e}_j - \sum_{i,j \in \{1,k,l\} \text{ or }
  i,j \not \in \{1,k,l\}} \mathbf{e}_{1kl}$.  
In this basis, for 
\small$\pi=(1m)(ij)(kl)$ \normalsize, we have $Q_{\pi} = \sum_{i\in [6]}\mathbf{e}_i-2\sum_{r\in\{i,j\}\text{ and }s\in \{k,l\}}\mathbf{e}_{1rs}$.
The $S_6$-action on
$\mathbb Z^{16}$ then just permutes the indices.  To check that this is indeed a description of $\Pic(\M)$ one only needs to check that the given
vectors obey the known relations between the boundary divisors.
Alternatively, we give the explicit change of basis from the Kapranov
basis for $\M$, coming from the description of $\M$ as the blow-up of
$\mathbb P^3$ at $5$ points and $10$ lines.  This has basis the
exceptional divisors $E_i, E_{jk}$ for $1 \leq i \leq 5$, $1 \leq j<k
\leq 5$, and the pull-back $\psi_6 = H$ of the hyperplane class on
$\mathbb P^3$.
 The change of basis between
these two bases is given by the formulas: 
\begin{equation*}
\begin{split}
\mathbf{e}_i&=\frac{1}{2}\big(H+E_i-\sum_{j\in [5]\setminus\{i\}}E_j+\frac{1}{2}(\sum_{j\in [5]\setminus\{i\}}E_{ij}-\sum_{j,k\in [5]\setminus\{i\}}E_{jk}\quad)\big),\quad \text{if } i<6,\\
\mathbf{e}_6&=\frac{1}{2}\big(-H+\sum_{i\in[5]}E_i+\frac{1}{2}\sum_{i,j\in[5]}E_{ij}\big),
\end{split}
\end{equation*}
and for $\{i,j\} \subset \{2,\dots ,6\}$, and $i<j$,
\[
\mathbf{e}_{1ij} =
\begin{cases}
1/2 E_{kl} & \text{if } j<6\text{ and }\{1,i,j,k,l\}=[5]  \\
1/2 E_{1i} & \text{if } j=6
\end{cases}.
\]

Let $A$ be the $16\times
40$-matrix with columns indexed by the boundary divisors $\delta_I$
and $Q_\pi$, in the order
\[
\{12\},\dots,\{56\},\{123\},\dots ,\{156\},(12)(34)(56),\dots ,(16)(25)(34),
\]
and such that $A_I=[\delta_I]$ and $A_\pi=[Q_\pi] \in \Pic(\M)$ in this basis.  
Explicitly, $A$ has the form
\[
A=(A_{\text{bnd}}|A_{KV}) 
\]
where $A_{bnd}$ is the matrix 
\begin{equation}
\label{e:matrixGb}
A_{bnd}=\left(
\begin{smallmatrix}
1&1&1&1&1&0&0&0&0&0&0&0&0&0&0&0&0&0&0&0&0&0&0&0&0\\
1&0&0&0&0&1&1&1&1&0&0&0&0&0&0&0&0&0&0&0&0&0&0&0&0\\
0&1&0&0&0&1&0&0&0&1&1&1&0&0&0&0&0&0&0&0&0&0&0&0&0\\
0&0&1&0&0&0&1&0&0&1&0&0&1&1&0&0&0&0&0&0&0&0&0&0&0\\
0&0&0&1&0&0&0&1&0&0&1&0&1&0&1&0&0&0&0&0&0&0&0&0&0\\
0&0&0&0&1&0&0&0&1&0&0&1&0&1&1&0&0&0&0&0&0&0&0&0&0\\
{-1}&{-1}&0&0&0&{-1}&0&0&0&0&0&0&{-1}&{-1}&{-1}&2&0&0&0&0&0&0&0&0&0\\
{-1}&0&{-1}&0&0&0&{-1}&0&0&0&{-1}&{-1}&0&0&{-1}&0&2&0&0&0&0&0&0&0&0\\
{-1}&0&0&{-1}&0&0&0&{-1}&0&{-1}&0&{-1}&0&{-1}&0&0&0&2&0&0&0&0&0&0&0\\
{-1}&0&0&0&{-1}&0&0&0&{-1}&{-1}&{-1}&0&{-1}&0&0&0&0&0&2&0&0&0&0&0&0\\
0&{-1}&{-1}&0&0&0&0&{-1}&{-1}&{-1}&0&0&0&0&{-1}&0&0&0&0&2&0&0&0&0&0\\
0&{-1}&0&{-1}&0&0&{-1}&0&{-1}&0&{-1}&0&0&{-1}&0&0&0&0&0&0&2&0&0&0&0\\
0&{-1}&0&0&{-1}&0&{-1}&{-1}&0&0&0&{-1}&{-1}&0&0&0&0&0&0&0&0&2&0&0&0\\
0&0&{-1}&{-1}&0&{-1}&0&0&{-1}&0&0&{-1}&{-1}&0&0&0&0&0&0&0&0&0&2&0&0\\
0&0&{-1}&0&{-1}&{-1}&0&{-1}&0&0&{-1}&0&0&{-1}&0&0&0&0&0&0&0&0&0&2&0\\
0&0&0&{-1}&{-1}&{-1}&{-1}&0&0&{-1}&0&0&0&0&{-1}&0&0&0&0&0&0&0&0&0&2
\end{smallmatrix}\right)
\end{equation}
and 
\begin{equation}
\label{e:matrixKV}
A_{KV}=\left(
\begin{smallmatrix}1&1&1&1&1&1&1&1&1&1&1&1&1&1&1\\
1&1&1&1&1&1&1&1&1&1&1&1&1&1&1\\
1&1&1&1&1&1&1&1&1&1&1&1&1&1&1\\
1&1&1&1&1&1&1&1&1&1&1&1&1&1&1\\
1&1&1&1&1&1&1&1&1&1&1&1&1&1&1\\
1&1&1&1&1&1&1&1&1&1&1&1&1&1&1\\
0&0&0&0&0&0&0&{-2}&{-2}&0&{-2}&{-2}&0&{-2}&{-2}\\
0&0&0&0&{-2}&{-2}&0&0&0&{-2}&0&{-2}&{-2}&0&{-2}\\
0&0&0&{-2}&0&{-2}&{-2}&0&{-2}&0&0&0&{-2}&{-2}&0\\
0&0&0&{-2}&{-2}&0&{-2}&{-2}&0&{-2}&{-2}&0&0&0&0\\
0&{-2}&{-2}&0&0&0&0&0&0&{-2}&{-2}&0&{-2}&{-2}&0\\
{-2}&0&{-2}&0&0&0&{-2}&{-2}&0&0&0&0&{-2}&0&{-2}\\
{-2}&{-2}&0&0&0&0&{-2}&0&{-2}&{-2}&0&{-2}&0&0&0\\
{-2}&{-2}&0&{-2}&{-2}&0&0&0&0&0&0&0&0&{-2}&{-2}\\
{-2}&0&{-2}&{-2}&0&{-2}&0&0&0&0&{-2}&{-2}&0&0&0\\
0&{-2}&{-2}&0&{-2}&{-2}&0&{-2}&{-2}&0&0&0&0&0&0
\end{smallmatrix}\right)
\end{equation}

We then get the following description of the matrix $R$ that induces
the morphism $\phi$ of Proposition~\ref{p:CoxRingRelations}.

\begin{lemma}  \label{l:Rdefinition}
Let $R$ be the matrix with columns indexed by $\mathcal I$ and $\Pi$ ordered as before, and rows indexed by the generators $z_{ij},u_\pi$ of $\C[Y]$, and entries 
\begin{equation}
\label{Rdefinition}
\begin{split}
R_{ij,I}&=\ord_{\delta_I}(z_{ij}),\\
R_{ij,\pi}&=\ord_{Q_\pi}(z_{ij}),\\
R_{\pi,I}&=\ord_{\delta_I}(u_\pi),\\
R_{\pi_1,\pi_2}&=\ord_{Q_{\pi_2}}(u_{\pi_1}).
\end{split}
\end{equation}
Then
\begin{equation}
\label{e:matrixR}
R=
\left(\begin{smallarray}{ccccccccccccccccccccccccc|ccccccccccccccc}
0& 1& -1& 0& 0& -1& 1& 0& 0& 0& 0& 0& 0& 0& 0& 0& 0& 0& 0& 0& 1& 1& -1& -1& 0& 0& 0& 0& 0& 0& 0& 0& 0& 0& 0& 0& 0& 0& 0& 0\\
 0& 1& 0& -1& 0& -1& 0& 1& 0& 0& 0& 0& 0& 0& 0& 0& 0& 0& 0& 1& 0& 1& -1& 0& -1& 0& 0& 0& 0& 0& 0& 0& 0& 0& 0& 0& 0& 0& 0& 0\\
 0& 1& 0& 0& -1& -1& 0& 0& 1& 0& 0& 0& 0& 0& 0& 0& 0& 0& 0& 1& 1& 0& 0& -1& -1& 0& 0& 0& 0& 0& 0& 0& 0& 0& 0& 0& 0& 0& 0& 0\\
 1& 0& -1& 0& 0& -1& 0& 0& 0& 1& 0& 0& 0& 0& 0& 0& 0& 1& 1& 0& 0& 0& -1& -1& 0& 0& 0& 0& 0& 0& 0& 0& 0& 0& 0& 0& 0& 0& 0& 0\\
 1& 0& 0& -1& 0& -1& 0& 0& 0& 0& 1& 0& 0& 0& 0& 0& 1& 0& 1& 0& 0& 0& -1& 0& -1& 0& 0& 0& 0& 0& 0& 0& 0& 0& 0& 0& 0& 0& 0& 0\\
 1& 0& 0& 0& -1& -1& 0& 0& 0& 0& 0& 1& 0& 0& 0& 0& 1& 1& 0& 0& 0& 0& 0& -1& -1& 0& 0& 0& 0& 0& 0& 0& 0& 0& 0& 0& 0& 0& 0& 0\\
 1& 1& -1& -1& 0& -1& 0& 0& 0& 0& 0& 0& 1& 0& 0& 1& 0& 0& 1& 0& 0& 1& -1& -1& -1& 0& 0& 0& 0& 0& 0& 0& 0& 0& 0& 0& 0& 0& 0& 0\\
 1& 1& -1& 0& -1& -1& 0& 0& 0& 0& 0& 0& 0& 1& 0& 1& 0& 1& 0& 0& 1& 0& -1& -1& -1& 0& 0& 0& 0& 0& 0& 0& 0& 0& 0& 0& 0& 0& 0& 0\\
 1& 1& 0& -1& -1& -1& 0& 0& 0& 0& 0& 0& 0& 0& 1& 1& 1& 0& 0& 1& 0& 0& -1& -1& -1& 0& 0& 0& 0& 0& 0& 0& 0& 0& 0& 0& 0& 0& 0& 0\\
 \noalign{\vskip1pt\hrule\vskip1pt}
 1& 1& -1& -1& -1& -2& 0& 0& 0& 0& 0& 0& 0& 0& 0& 0& 0& 0& 0& 0& 1& 1& -1& -1& -2& 1& 0& 0& 0& 0& 0& 0& 0& 0& 0& 0& 0& 0& 0& 0\\
 1& 1& -1& -1& -1& -2& 0& 0& 0& 0& 0& 0& 0& 0& 0& 0& 0& 0& 0& 1& 0& 1& -1& -2& -1& 0& 1& 0& 0& 0& 0& 0& 0& 0& 0& 0& 0& 0& 0& 0\\
 1& 1& -1& -1& -1& -2& 0& 0& 0& 0& 0& 0& 0& 0& 0& 0& 0& 0& 0& 1& 1& 0& -2& -1& -1& 0& 0& 1& 0& 0& 0& 0& 0& 0& 0& 0& 0& 0& 0& 0\\
 1& 1& -1& -1& -1& -2& 0& 0& 0& 0& 0& 0& 0& 0& 0& 0& 0& 1& 1& 0& 0& 0& -1& -1& -2& 0& 0& 0& 1& 0& 0& 0& 0& 0& 0& 0& 0& 0& 0& 0\\
 1& 1& -1& -1& -1& -2& 0& 0& 0& 0& 0& 0& 0& 0& 0& 0& 1& 0& 1& 0& 0& 0& -1& -2& -1& 0& 0& 0& 0& 1& 0& 0& 0& 0& 0& 0& 0& 0& 0& 0\\
 1& 1& -1& -1& -1& -2& 0& 0& 0& 0& 0& 0& 0& 0& 0& 0& 1& 1& 0& 0& 0& 0& -2& -1& -1& 0& 0& 0& 0& 0& 1& 0& 0& 0& 0& 0& 0& 0& 0& 0\\
 1& 1& -1& -1& -1& -2& 0& 0& 0& 0& 0& 0& 0& 0& 0& 0& 0& 1& 1& 0& 1& 1& -2& -2& -2& 0& 0& 0& 0& 0& 0& 1& 0& 0& 0& 0& 0& 0& 0& 0\\
 1& 1& -1& -1& -1& -2& 0& 0& 0& 0& 0& 0& 0& 0& 0& 1& 0& 0& 1& 0& 1& 0& -2& -2& -1& 0& 0& 0& 0& 0& 0& 0& 1& 0& 0& 0& 0& 0& 0& 0\\
 1& 1& -1& -1& -1& -2& 0& 0& 0& 0& 0& 0& 0& 0& 0& 1& 0& 1& 0& 0& 0& 1& -2& -2& -1& 0& 0& 0& 0& 0& 0& 0& 0& 1& 0& 0& 0& 0& 0& 0\\
 1& 1& -1& -1& -1& -2& 0& 0& 0& 0& 0& 0& 0& 0& 0& 0& 1& 0& 1& 1& 0& 1& -2& -2& -2& 0& 0& 0& 0& 0& 0& 0& 0& 0& 1& 0& 0& 0& 0& 0\\
 1& 1& -1& -1& -1& -2& 0& 0& 0& 0& 0& 0& 0& 0& 0& 1& 0& 0& 1& 1& 0& 0& -2& -1& -2& 0& 0& 0& 0& 0& 0& 0& 0& 0& 0& 1& 0& 0& 0& 0\\
 1& 1& -1& -1& -1& -2& 0& 0& 0& 0& 0& 0& 0& 0& 0& 1& 1& 0& 0& 0& 0& 1& -2& -1& -2& 0& 0& 0& 0& 0& 0& 0& 0& 0& 0& 0& 1& 0& 0& 0\\
 1& 1& -1& -1& -1& -2& 0& 0& 0& 0& 0& 0& 0& 0& 0& 0& 1& 1& 0& 1& 1& 0& -2& -2& -2& 0& 0& 0& 0& 0& 0& 0& 0& 0& 0& 0& 0& 1& 0& 0\\
 1& 1& -1& -1& -1& -2& 0& 0& 0& 0& 0& 0& 0& 0& 0& 1& 0& 1& 0& 1& 0& 0& -1& -2& -2& 0& 0& 0& 0& 0& 0& 0& 0& 0& 0& 0& 0& 0& 1& 0\\
 1& 1& -1& -1& -1& -2& 0& 0& 0& 0& 0& 0& 0& 0& 0& 1& 1& 0& 0& 0& 1& 0& -1& -2& -2& 0& 0& 0& 0& 0& 0& 0& 0& 0& 0& 0& 0& 0& 0& 1
\end{smallarray}\right)
\end{equation}
\end{lemma}
\begin{proof}
Write $R$ for the matrix with entries given by orders of vanishing,
and $\tilde{R}$ for the matrix on the right hand side of \eqref{e:matrixR}.  The $z_{ij}$ and
$u_{\pi}$ are rational functions on $\M$ whose divisors are supported
on the union of the boundary divisors $\delta_I$ and the $Q_{\pi}$. Note that $u_{\pi}=f_{\pi}$ in $\C[Y]$.
Since $\divv(z_{ij})$ and $\divv(f_{\pi})$ are principal divisors,
$\sum_{I} \ord_{\delta_I}(z_{ij}) \delta_I + \sum_{\pi} \ord_{Q_{\pi}}(z_{ij}) Q_{\pi}
= \sum_{I} \ord_{\delta_I}(f_{\pi})\delta_I + \sum_{\pi}
\ord_{Q_{\pi}}(f_{\pi})Q_{\pi} = 0 \in \Pic(\M)$.  The columns of the
matrix $A$ are representatives for the boundary divisors and the
Keel-Vermeire divisors in $\Pic(\M)$, so this implies that $AR^T =0$.  Since the rows
of $\tilde{R}$ are a basis for $\ker(A)$, every row of $R$ lies in the
rowspace of $\tilde{R}$.

Note that $\tilde{R}$ has the block form:
\[
\tilde{R}=\left(
\begin{array}{c|c}
R' &0 \\ \hline
C &I_{15}
\end{array}\right)
\]
where $I_{15}$ is the identity matrix of rank 15 and $R'$ is the $9
\times 25$ matrix appearing in Proposition 5.2 of
\cite{GibneyMaclaganEquations}.  We can further decompose $R' = (R'_1
| I_9 | R'_2)$, where $R'_1$ is $9 \times 6$, and $R'_2$ is $9 \times
10$, and $C = (C_1 | \mathbf{0} | C_2)$, where $C_1$ is $15 \times 6$,
$\mathbf{0}$ is the $15 \times 9$ matrix with every entry $0$, and
$C_2$ is $15 \times 10$.  An element of the rowspace of $\tilde{R}$
then has four blocks of coordinates, and is determined by its values
in the second and fourth blocks, which are indexed by the
$\delta_{ij}$ with $\{i,j\} \in \mathcal E$, and by the $Q_{\pi}$ with
$\pi \in \Pi$.

To show that $R = \tilde{R}$ it thus suffices to show the following four facts:
\begin{enumerate}
\item $\ord_{\delta_{ij}}(z_{kl}) = 1$ if $ij=kl$, and $0$ otherwise,
  for $\{i,j\}, \{k,l\} \in \mathcal E$,
\item $\ord_{\delta_{ij}}(f_{\pi}) = 0$ for 
 $\{i,j\} \in \mathcal E$ and $\pi \in \Pi$,
\item $\ord_{Q_{\pi}}(z_{ij}) = 0$ for 
 $\{i,j\} \in \mathcal E$ and $\pi \in \Pi$, and
\item $\ord_{Q_{\pi}}(f_{\pi'}) = 1$ if $\pi = \pi'$, and $0$ otherwise,
  for $\pi, \pi' \in \Pi$.
\end{enumerate}

Since $Q_{\pi}$ intersects $M_{0,6}$, the last two follow from the
fact that $f_{\pi}$ is an equation for $Q_{\pi}$, and no $z_{ij}$ with
$\{i,j\} \in \mathcal E$ or $f_{\pi'}$ with $\pi' \neq \pi$ lies in the ideal $I_6 + \langle f_{\pi} \rangle$.

To see the first two, we use the description of \cite[Theorem
  5.7]{GibneyMaclaganEquations}, where it is shown that $\M$ is the
closure of $M_{0,6} \subset (\mathbb C^*)^9$ in a toric variety $X_{\Delta}$ given
by a fan $\Delta$ whose rays are spanned by the columns of $R'$.  The
intersection of $\M$ with the torus-invariant divisor corresponding
to the ray indexed by $\delta_I$ is $\delta_I$.  For $\{i,j \} \in
\mathcal E$ the ray indexed by $\delta_{ij}$ is the basis vector
$\mathbf{e}_{ij}$.  Thus the variety of $\overline{I}_6 := I_6 \cap
\C[z_{ij}, z_{kl}^{\pm 1} : \{k,l \} \neq \{i,j\}]$ in $\mathbb A^1
\times T^8$ is the union of $M_{0,6}$ and an open part of
$\delta_{ij}$.  This means that to show that $\ord_{\delta_{ij}}(g)=0$
for $g \in \C[z_{ij}, z_{kl}^{\pm 1}]/\overline{I}_6$ it suffices to
show that $g \not \in \overline{I}_6 + \langle z_{ij} \rangle$.  This
is the case for $g=z_{kl}$ with $\{k,l \} \neq \{i,j\}$, and for
$g=f_{\pi}$.  Since in addition we have $\ord_{\delta_{ij}}(z_{ij}) =
1$, this finishes the proof.
 \end{proof}

We will make extensive computational use of the $S_6$-action on
$\Cox(\M)$.  We first review how this works for automorphisms of Mori
Dream Spaces in more generality.

Let $X$ be a Mori Dream Space with Cox ring $S/I$, where $S = \mathbb
C[x_1,\dots,x_s]$.  Suppose that $G$ is a group acting on $S$ by
monomial maps: $g \cdot x_i = \alpha x^u$ for some $\alpha \in \mathbb
C$ and monomial $x^u$ with $\deg(x^u) = \deg(x_i)$.  Write $\mathbf{x}
= \prod_{i=1}^s x_i$.  The assumption that $S$ acts by monomial maps
means that the action of $G$ on $S$ extends to an action of $G$ on
$S_{\mathbf{x}}$, which then restricts to an action on the Laurent
polynomial ring $(S_{\mathbf{x}})_{\mathbf{0}}$.  Write $X^0 =
\Spec(((S/I)_{\mathbf{x}})_{\mathbf{0}})$ for the complement in $X$ of
the divisors of the $x_i$ for $1 \leq i \leq s$.  In the next lemma we
show that if the action restricts to a given action of $G$ on $X^0$,
then the $G$ action on $S$ fixes $I$.

\begin{lemma} \label{l:Gaction}
Let $X$ be a Mori dream space with Cox ring $\Cox(X) \cong S/I$ where
$S = \mathbb C[x_1,\dots,x_s]$.  Suppose that a group $G$ acts by 
degree-preserving monomial maps on $S$.  Then the ideal $I$ is invariant under the action of
$G$ if the restriction of the action of $G$ to
$(S_{\mathbf{x}})_{\mathbf{0}}$ preserves
$I^0=(IS_{\mathbf{x}})_{\mathbf{0}}$.
\end{lemma}

\begin{proof}
It suffices to show that $g \cdot f \in I$ for all $f \in I$.  That
shows that $g I \subset I$, and the reverse inclusion then follows
from applying $g^{-1}$.  For $f \in I^0S_{\mathbf{x}} \cap S$,
we have $f = \sum f_i h_i/\mathbf{x}^{m_i}$ for some $h_i \in S$
and $f_i \in I^0$.  Then $g \cdot f = \sum (g \cdot f_i) (g \cdot
h_i)/g \cdot \mathbf{x}^{m_i} = \sum (g\cdot f_i) (g \cdot h_i) /( g
\cdot \mathbf{x}^{m_i}) = \sum f'_i h'_i/\mathbf{x}^{n_i}$, where
$f'_i \in I_0$, by hypothesis, and $h'_i \in S$.  Here we use that the
action of $G$ on $S$ is via monomial maps to know that $g \cdot
\mathbf{x}^{m_i}$ is again a scalar multiple of a  monomial.  Since $G$ acts on $S$, we have
$g \cdot f \in S$, so this shows that $g \cdot f \in
I^0S_{\mathbf{x}} \cap S$.  As this equals $I$ by
Proposition~\ref{p:CoxRingRelations}, we have the desired result.
\end{proof}

Note that the induced map $\mathbb C[z_{ij}^{\pm 1}] \rightarrow
\mathbb C[x_I^{\pm 1}, y_{\pi}^{\pm 1} : I \in \mathcal I, \pi \in \Pi]$ given by the
matrix \eqref{e:matrixR} has the property that $\phi(z_{ij})$ is a
monomial in the variables $x_{ijk}$ times the formula given in
\eqref{eqtn:zij}.

We will apply Lemma~\ref{l:Gaction} to the $S_6$-action on $\M$.  Let
$S = \mathbb C[x_I, y_{\pi} : I \in \mathcal I, \pi \in \Pi]$.  We
have the isomorphism $(S_{\mathbf{x}})_{\mathbf{0}} \cong \mathbb
C[z_{24}^{\pm 1},\dots,z_{56}^{\pm 1}]$, where $S_6$ acts on $\mathbb
C[z_{24}^{\pm 1},\dots, z_{56}^{\pm 1}]$ as in
Table~\ref{t:S6action1}. 
Consider the $S_6$-action on $S$ defined as follows.  We set $\sigma
\cdot x_{ij} = x_{\sigma(i) \sigma(j)}$, with the convention that
$x_{ji} = -x_{ij}$.  This can be regarded as induced from the natural
action of $S_6$ on $\wedge^2 \mathbb Z^6$.  Similarly, the $S_6$-action on the variables $x_{ijk}$ is induced from the natural $S_6$-action on $\Sym^2(\wedge^3 \mathbb Z^6)$.  Here the $\Sym^2$ arises
from the fact that $x_{123}=x_{456}$.  Explicitly, this means that, for
example, $(12) x_{123} = -x_{123}$, $(12) x_{456} = -x_{456}$, and
$(12)x_{145}= x_{245}$.  It is then a straightforward computation to
check that $\phi(\sigma \cdot z_{ij}) = \sigma \cdot \phi(z_{ij})$ for
all $\sigma \in S_6$.  Since the action of $S_6$ on the variables
$z_{ij}$ is induced from the action on $x_{ij}$ using
\eqref{eqtn:zij}, it is only necessary for this to check that this
action is correct on the variables $x_{ijk}$.

The action of $S_6$ on the variables $y_{\pi}$ is then uniquely
determined by the requirement that $\sigma \cdot \phi(u_{\pi}) =
\phi(\sigma \cdot u_{\pi})$, as $\phi(u_{\pi}) = m y_{\pi}$ for a
monomial $m$ in the variables $x_I$.  Explicitly, this action is
induced from the natural $S_6$-action on $\wedge^3 (\Sym^2 \mathbb
Z^6)$.  For example, $(12) y_{(12)(34)(56)} = y_{(12)(34)(56)}$ since
$(12) ((\mathbf{e}_1+\mathbf{e}_2) \wedge (\mathbf{e}_3+\mathbf{e}_4)
\wedge (\mathbf{e}_5 + \mathbf{e}_6)) = (\mathbf{e}_1+\mathbf{e}_2)
\wedge (\mathbf{e}_3+\mathbf{e}_4) \wedge (\mathbf{e}_5 +
\mathbf{e}_6)$.  Note that a careful choice of signs in
Table~\ref{t:QpiEqns} was required for the action to work in this
fashion.

\section{The relations}
\label{s:maintheorem}

In this section we apply the method of
Proposition~\ref{p:CoxRingRelations} to compute the relations in
$\Cox(\overline{M}_{0,6})$.  By Castravet's theorem
(Theorem~\ref{t:Castravet}) this Cox ring is generated by sections
$x_I$ of the $25$ boundary divisors $\delta_{I}$, and by sections
$y_{\pi}$ of the $15$ Keel-Vermeire divisors $Q_{\pi}$.  The main
theorem of this section describes the ideal of relations between these
$40$ generators.

{
\renewcommand{\thetheorem}{\ref{t:maintheorem}}

\begin{theorem} 
Let $S = \mathbb C[x_I, y_{\pi} : I \in \mathcal I, \pi \in \Pi]$.  The Cox ring
of $\M$ is $S/I_{\M}$, where $I_{\M}$ is an ideal with 225 generators,
which come in five symmetry classes:

\begin{enumerate}
\item  \label{item:case1} $15$ of the form:
\begin{equation*}
x_{ij}x_{kl}x_{ijn}x_{kln}-x_{ik}x_{jl}x_{ikn}x_{jln}+x_{il}x_{jk}x_{iln}x_{jkn}
\end{equation*}
for $1\leq i<j<k<l\leq 6$, $m<n$ and $\{i,j,k,l,m,n\}=\{1,2,3,4,5,6\}$;

\item  \label{item:case2} $60$ of the form:
\begin{equation*}
x_{1ik}y_{(1m)(ij)(kl)}+x_{1j}x_{mk}x_{il}x_{1mj}x_{1mk}x_{1ij}+ x_{1l}x_{mi}x_{jk}x_{1mi}x_{1ml}x_{1kl},
\end{equation*}
for $\{i,j,k,l,m\}=\{2,3,4,5,6\}$;

\item \label{item:case3} $45$ of the form:
\begin{equation*}
x_{ij}y_{(ij)(kl)(mn)} +  x_{ik}x_{il}x_{jm}x_{jn}x_{ikl}^2 -  x_{im}x_{in}x_{jl}x_{jk}x_{imn}^2,
\end{equation*}
for $\{i,j,k,l,m,n\}=\{1,2,3,4,5,6\}$ and $i<j$;

\item \label{item:case4} $45$ of the form:
\begin{equation*}
  y_{(ij)(kl)(mn)}y_{(ij)(km)(ln)}-x_{il}x_{lj}x_{jm}x_{mi}x_{nk}^{2}x_{ijm}^{2}x_{ijl}^{2} +x_{in}x_{nj}x_{jk}x_{ki}x_{ml}^{2}x_{ijk}^{2}x_{ijn}^{2},
\end{equation*}
for $\{i,j,k,l,m,n\} = \{1,2,3,4,5,6\}$;

\item \label{item:case5} $60$ of the form:
\begin{multline*}
\qquad \quad y_{(1i)(jk)(lm)}y_{(1j)(il)(km)} - x_{1k}x_{1l}x_{ik}x_{im}x_{jl}x_{jm}x_{1ik}x_{1il}x_{1jk}x_{1jl}\\
+x_{1k}x_{1l}x_{ij}x_{im}x_{jm}x_{kl}x_{1ij}x_{1il}x_{1jk}x_{1kl} -x_{1k}x_{1m}x_{ij}x_{im}x_{jl}x_{kl}x_{1ij}x_{1im}x_{1jk}x_{1km}\\
-x_{1l}x_{1m}x_{ij}x_{ik} x_{jm} x_{kl} x_{1ij} x_{1il} x_{1jm} x_{1lm},
\end{multline*}
for $\{i,j,k,l,m\} = \{2,3,4,5,6\}$.
\end{enumerate}
Here we follow the convention that $x_{ij}\!=\!-x_{ji}$,
$x_{ijk}\!=\!-x_{jik}=\!-x_{ikj}$, $x_{ijk}\!=\!x_{lmn}$ for
\small$\{i,j,k,l,m,n \}= \{1,2,3,4,5,6\}$\normalsize, with $i<j<k$ and
$l<m<n$, and $y_{(ij)(kl)(mn)}=-y_{(kl)(ij)(mn)}=-y_{(ij)(mn)(kl)}$.
\end{theorem}
}

\begin{proof}
By  Example~\ref{e:M06embedding} $\M$ has an embedding
$i:\M \hookrightarrow X_{\Sigma}$ into a $24$-dimensional toric 
variety in such a way that the intersection of $\M$ with the 
torus $T^{24}$ of $X_{\Sigma}$ is $Y=M_{0,6} \setminus \cup_{\pi \in \Pi} Q_{\pi}$.

By Lemma~\ref{l:idealY} we have
\[
\C[Y]=\C[z_{ij}^{\pm 1},u_\pi^{\pm 1}: 2\leq i<j \leq 6, ij \neq 23, \pi \in \
\Pi]/J_{KV},
\]
where $J_{KV}=I_6+\left< u_\pi-f_\pi: \pi \in \Pi \right> \subset
\C[T^{24}]$.  Here $I_6$ and $f_{\pi}$ are as defined in
Lemma~\ref{l:M06equations} and Table~\ref{t:QpiEqns}.

We denote by $\mathbf{x}\mathbf{y}$ the product $\prod_{I \in \mathcal I}x_I
\prod_{\pi \in \Pi} y_{\pi}$.  Let $\phi$ be given by
\begin{equation*}
\label{e:phi}
\begin{split}
\phi:\C[z_{ij}^{\pm 1},u_\pi^{\pm 1}]  \rightarrow \big(S_{\mathbf {xy}}\big)_0\\
 z_{ij}  \mapsto \prod_{I}x_I^{R_{{ij},I}},\quad
 \\
 u_\pi  \mapsto y_\pi\prod_{I}x_I^{R_{\pi,I}},
\end{split}
\end{equation*}
where $R$ is the matrix defined in (\ref{e:matrixR}) and $R_{{ij},I} $
(respectively $R_{{ij},\pi}$) denotes the entry with row labelled
${ij}$ and column labelled $I$ (respectively $\pi$).

Proposition~\ref{p:CoxRingRelations} then implies that the ideal
$I_{\M}$ equals
$$I_{\M} = (\phi(J_{KV})S_{\mathbf{x}\mathbf{y}})  \cap S.$$

In order to compute the ideal of relations we thus only need to
compute $\phi(I(Y))S_{\mathbf{xy}} \cap S$. We will exhibit the
different symmetry types as elements of $\phi(I(Y))S_{\mathbf{xy}}
\cap S$ and then argue that the ideal generated by these equations is
saturated with respect to the product of the variables, so equals
$I_{\M}$.

\begin{description}
\item[(1)] 

For $j,k,l$ with $[6]=\{1,2,3,j,k,l\}$ we have
\[
\phi_R(z_{2j})=\frac{x_{2j}x_{13}x_{13k}x_{13l}}{x_{1j}x_{23}x_{1jk}x_{1jl}},
\]
\[
\phi_R(z_{3j})=\frac{x_{3j}x_{12}x_{12k}x_{12l}}{x_{1j}x_{23}x_{1jk}x_{1jl}}
\]
and
\[
\phi_R(z_{jk})=\frac{x_{jk}x_{12}x_{13}x_{123}x_{12l}x_{13l}}{x_{1j}x_{1k}x_{23}x_{1jk}x_{1jl}x_{1kl}}.
\]
Then, 
\begin{align*}
\phi_R(z_{34}+1-z_{24})&=\frac{x_{12}x_{34}x_{125}x_{126}}{x_{14}x_{23}x_{145}x_{146}}-\frac{x_{13}x_{24}x_{135}x_{136}}{x_{14}x_{23}x_{145}x_{146}}+1\\
&=\frac{x_{12}x_{34}x_{125}x_{126}-x_{13}x_{24}x_{135}x_{136}+x_{14}x_{23}x_{145}x_{146}}{x_{14}x_{23}x_{145}x_{146}}.
\end{align*}
 Thus the polynomial
\begin{equation} \label{eqtn:case1}
x_{12}x_{34}x_{125}x_{126}-x_{13}x_{24}x_{135}x_{136}+x_{14}x_{23}x_{145}x_{146}
\end{equation}
times a Laurent monomial in the variables $x_I$ is in $\phi_R(I(Y))$
and therefore it itself is in
$\big(\phi_R(I(Y))S_{\mathbf{xy}}\big)\cap S$.  This is the case $i=1,
j=2, k=3, l=4, m=5, n=6$ of case \ref{item:case1}.  The $S_6$-orbit of this
polynomial has size $15$; this can be checked directly by hand, or
using the accompanying {\tt Macaulay2} package available from the
authors' webpages \cite{M06M2package}.

\item[(2)]
Let $\pi=(12)(34)(56)$. Then $u_\pi -f_\pi=u_\pi-z_{24}+z_{25} z_{26}$ and 
\begin{align*}
\qquad \quad h&=u_\pi+ z_{45} +z_{25} z_{36}\\
                &=(u_{\pi}-f_{\pi}) +z_{25}(z_{36}+1-z_{26})+(z_{45}-z_{25}+z_{24})\quad \qquad
\end{align*}
is an element of $I(Y)$. Therefore
\begin{multline*}
\qquad \quad\phi_R(h)=\Big(\frac{x_{12}x_{13}x_{136}}{x_{14}x_{15}x_{16}x_{23}^{2}x_{145}x_{146}x_{156}^{2}}\Big)\cdot\\
(x_{135}y_{\pi}+x_{14}x_{25}x_{36}x_{124}x_{125}x_{134}+x_{16}x_{23}x_{45}x_{123}x_{126}x_{156})\qquad
\end{multline*}
and hence
\begin{equation} \label{eqtn:case2}
x_{135}y_{(12)(34)(56)}+x_{14}x_{25}x_{36}x_{124}x_{125}x_{134}+x_{16}x_{23}x_{45}x_{123}x_{126}x_{156}
\end{equation}
is an element of $\big(\phi_R(I(Y))S_{\mathbf{xy}}\big)\cap S$.  This
the case $i=3, j=4, k=5, l=6, m=2$ of case~\ref{item:case2}.  The
$S_6$-orbit has size $60$; again this can be checked by hand, or using the
accompanying package.

\item[(3)] 
For $\pi=(12)(34)(56)$, we have that
\begin{multline*}
\qquad \quad\phi_R(u_\pi-f_\pi)=\Big(\frac{x_{13}x_{135}x_{136}}{x_{14}x_{15}x_{16}x_{23}^{2}x_{145}x_{146}x_{156}^{2}}\Big)\cdot\\
(x_{12}y_{\pi}+x_{13}x_{14}x_{25}x_{26}x_{134}^{2}-x_{15}x_{16}x_{23}x_{24}x_{156}^{2})\qquad \quad
\end{multline*}
and therefore
\begin{equation} \label{eqtn:case3}
x_{12}y_{\pi}+x_{13}x_{14}x_{25}x_{26}x_{134}^{2}-x_{15}x_{16}x_{23}x_{24}x_{156}^{2}
\end{equation}
is an element of $\big(\phi_R(I(Y))S_{\mathbf{xy}}\big)\cap S$.  This
is the case $i=1, j=2, k=3, l=4, m=5, n=6$ of case~\ref{item:case3}.
The $S_6$-orbit has size $45$.

\item[(4)]
Let $\pi_1=(12)(34)(56)$ and $\pi_2=(12)(35)(46)$. 
Let
\[
h=y_{\pi_1}y_{\pi_2}-x_{14}x_{15}x_{24}x_{25}x_{36}^{2}x_{124}^{2}x_{125}^{2}+x_{13}x_{16}x_{23}x_{26}x_{45}^{2}x_{123}^{2}x_{126}^{2}.
\]

Let $F_1$ be the polynomial \eqref{eqtn:case1},
let $F_2$ be the
polynomial~\eqref{eqtn:case2}, and let $F_3$ be the polynomial of \eqref{eqtn:case3}.  
Set $g = F_3  ((45) \cdot F_3)$.  

 We claim that 
$x_{12}^2h - g$ lies in the ideal 
\[\langle  (35) \cdot F_1, (46) \cdot F_1, (36)\cdot F_2, (45) \cdot F_2, (34)(56) \cdot F_2, F_2 \rangle \subseteq
\big(\phi_R(I(Y))S_{\mathbf{xy}}\big)\cap S.
\]
This can in principle be checked by hand, but is also confirmed in the accompanying {\tt Macaulay2} package.   Since $F_3 \in 
\big(\phi_R(I(Y))S_{\mathbf{xy}}\big)\cap S$, we also have $g \in 
\big(\phi_R(I(Y))S_{\mathbf{xy}}\big)\cap S$.  Thus 
$x_{12}^2h$ is also in $\phi_R(Y)S_{\mathbf{xy}}$ and so $h \in
 I_{\M}$.  This is the case $i=1, j=2, k=3,  l=4,  m=5, n=6$ of case~\ref{item:case4}.  The $S_6$-orbit has size $60$.

\item[(5)]
 Finally, let 
$\pi_1=(12)(34)(56)$ and $\pi_2=(13)(25)(46)$.  Set
\begin{multline*}
\qquad \quad h=y_{\pi_1}y_{\pi_2}-x_{14}x_{15}x_{24}x_{26}x_{35}x_{36}x_{124}x_{125}x_{134}x_{135}\\
+x_{14}x_{15}x_{23}x_{26}x_{36}x_{45}x_{123}x_{125}x_{134}x_{145}\\\qquad-x_{14}x_{16}x_{23}x_{26}x_{35}x_{45}x_{123}x_{126}x_{134}x_{146}\\
-x_{15}x_{16}x_{23}x_{24} x_{36} x_{45} x_{123} x_{125} x_{136} x_{156}.
\end{multline*}
Set $g = F_3((23)(45) \cdot F_3) \in 
\phi_R(I(Y))S_{\mathbf{xy}}$.
We claim that $x_{12}x_{13}h-g$ lies in the ideal
\begin{multline*}
\qquad \langle (25) \cdot F_1, (15) \cdot F_1, (25)(46) \cdot F_1, (35) \cdot F_1,  (46) \cdot F_1, (36) \cdot F_1, \\  (236) \cdot F_2, (12) \cdot F_2, (154) \cdot F_3, (13)(24) \cdot F_3 \rangle.\\
\end{multline*}
Thus $x_{12}x_{13}h \in \phi_R(I(Y))S_{\mathbf{xy}}$ and therefore
$h\in I_{\M}$.  This is the case $i=2, j=3, k=4, l=5, m=6$ of
case~\eqref{item:case5}.  The $S_6$-orbit has size $60$.
\end{description}
 This shows that one representative of each $S_6$-orbit from the
statement of the theorem lies in
$(\phi(J_{KV})S_{\mathbf{x}\mathbf{y}}) \cap S$.
Lemma~\ref{l:Gaction} and the explicit action of $S_6$ given at the
end of Section~\ref{s:coxequations} then imply that all $225$
polynomials given in the theorem statement are in $I_{\M}=
(\phi(J_{KV})S_{\mathbf{x}\mathbf{y}}) \cap S$.  To show that these
polynomials in fact generate $I_{\M}$, we first note that the ideal
generated by these $225$ generators in $S_{\mathbf{x}\mathbf{y}}$
equals $\phi(J_{KV})S_{\mathbf{x}\mathbf{y}}$.  This is the case
because the polynomial of case~\ref{item:case1} is a monomial multiple of the image under
$\phi$ of a representative of one $S_6$-orbit of generators of
$J_{KV}$: $z_{34}+1-z_{24}$. The polynomial of case~\ref{item:case2} is a monomial
multiple of the image under $\phi$ of the other $S_6$-orbit of
generators: $y_{\pi}-f_{\pi}$.  It thus suffices to check that the
ideal generated by the given $225$ generators is saturated with
respect to the product of the variables.  This can be done,
for example, using the accompanying {\tt Macaulay2}
package.
\end{proof}

\section{Small birational models}
\label{s:smallbirationalmodels}

A positive consequence of the fact that $\M$ is a Mori dream space
is that the movable cone of $\M$ decomposes into finitely many Mori
chambers, each of which is the pullback of the nef cone of a
birational model of $\M$, and each small $\mathbb Q$-factorial
modification of $\M$ occurs in this list.  These chambers are
precisely the GIT chambers of the GIT description
\begin{equation} \label{eqtn:bigGIT}
\M = V(I_{\M}) \git_{\alpha} H
\end{equation}
given in Example~\ref{e:M06embedding}.

The main result of this section is Proposition~\ref{p:Esuffices},
which gives a simpler GIT problem whose chambers also describe the
chamber decomposition of the movable cone.  The new GIT problem is the
quotient of a variety in $\mathbb A^{25}$ instead of $\mathbb A^{40}$,
which makes the computations significantly simpler.  An updated
version of \cite{BKR} (Remark 6.7) takes this decomposition into
account.

\begin{lemma} \label{l:Qpihyperplane}
Fix $\pi= (ij)(kl)(mn)$.  The hyperplane $\Delta_{\pi}$ in $\Pic(\M)$ spanned by
$\{ \delta_{ab} : \{a,b\} \not \in \{ \{i,j\}, \{k,l\},
\{m,n\} \} \cup \{ \delta_{abc} : i,j \in \{a,b,c\} \text{ or } k, l
\in \{a ,b ,c \} \text{ or } m,n \in \{a,b,c \} \}$ separates
$Q_{\pi}$ from all other boundary divisors and $Q_{\pi'}$. 
\end{lemma}

\begin{proof}
Consider the element $\rho=\sum_{i=1}^6 \mathbf{e}_{i}^* +
\mathbf{e}^*_{ikm} + \mathbf{e}^*_{ikn} + \mathbf{e}^*_{ilm} +
\mathbf{e}^*_{iln} \in \Pic(\M)^*$, where the $*$ denotes the dual
basis element in the symmetric basis for $\Pic(\M)$.  By examining the
matrices $A_{bnd}$ and $A_{KV}$ given in \eqref{e:matrixGb} and
\eqref{e:matrixKV}, we obtain $\rho(Q_{\pi}) = -2$,
$\rho(\delta_{ij})=\rho(\delta_{kl})=\rho(\delta_{mn}) = 2$,
$\rho(\delta_{ikm}) = \rho(\delta_{ikn}) = \rho(\delta_{ilm}) =
\rho(\delta_{iln}) = 2$, and all other $\rho(\delta_{I})=0$.  In
addition $\rho(Q_{\pi'}) \in \{2,3\}$ for all $\pi' \in \Pi$ with $\pi'
\neq \pi$.  Let $H$ be the hyperplane $\{ x \in \Pic(\M) \otimes
\mathbb R: \rho(x)=0 \}$.  Then $H$ contains the generators of
$\Delta_{\pi}$, so it remains to check that $\spann(\Delta_{\pi})$ is
fifteen dimensional, which can be done by computing the rank of the
corresponding $16 \times 18$ matrix.
\end{proof}

Let $J = I_{\M} \cap \mathbb C [x_I : I \in \mathcal I]$.  This is
generated by the $15$ polynomials of the first symmetry class of
Theorem~\ref{t:maintheorem}, and by the $S_6$-orbit, which also has
size $15$, of the polynomial $x_{12}x_{26}x_{34}x_{35}x_{45}x_{126}^2
- x_{13}x_{24}x_{25}x_{36}x_{45}x_{136}^2 +
x_{14}x_{23}x_{25}x_{35}x_{46}x_{146}^2 -
x_{15}x_{23}x_{24}x_{34}x_{56}x_{156}^2$.  The Picard torus $H$ acts
on the affine space $\mathbb A^{25}$ with coordinates $\{x_I : I \in
\mathcal I\}$.  Since $J$ is homogeneous with respect to the
induced grading, $H$ also acts on $V(J) \subset \mathbb A^{25}$.  By
\cite{GibneyMaclaganEquations}*{Theorem 7.1} we have
\begin{equation} \label{eqtn:newGIT} \M = V(J) \git_{\alpha} H 
\end{equation}  where $\alpha$ is any character of $H$
corresponding to an ample divisor on $\M$.
Let $E = \pos(\delta_I : I \in \mathcal I) \subseteq \Pic(\M)_{\mathbb
  R}$ be the cone generated by the boundary divisors.

\begin{proposition} \label{p:Esuffices}
The movable cone $\Mov(\M)$ of $\M$ is contained in the cone $E$.  The
GIT chambers of the GIT problem \eqref{eqtn:newGIT} that intersect
$\Mov(\M)$ contain the Mori chambers for $\M$,  and the corresponding
birational models equal the GIT quotients
$$ V(J) \git_{\beta} H$$
as $\beta$ varies over the relative interior of the chambers.
\end{proposition}

\begin{proof}
The movable cone for $\M$ is the movable cone of the toric variety
$X_{\Sigma}$.  For $\rho \in \mathcal I \cup \Pi$ we
write $D_{\rho}$ for $\delta_I$ if $\rho=I \in \mathcal I$ and for
$Q_{\pi}$ if $\rho = \pi \in \Pi$.  The movable cone is then 
$$\Mov(\M) = \cap_{\rho \in \mathcal I \cup \Pi } \pos(D_{\rho'} : \rho' \in
\mathcal I \cup \Pi, \rho' \neq \rho).$$

Write $S$ for the polynomial ring $\mathbb C[ x_I, y_{\pi} : I \in
  \mathcal I, \pi \in \Pi]$, and $S'$ for the polynomial ring $\mathbb
C[ x_I : I \in \mathcal I]$.  The inclusion $S'/J \rightarrow
S/I_{\M}$ induces an inclusion $$\phi_{\mathbf{v}} \colon \oplus_{\ell
  \geq 0} (S'/J)_{\ell \mathbf{v}} \rightarrow \oplus_{\ell \geq 0}
(S/I_{\M})_{\ell \mathbf{v}}$$ for any $\mathbf{v} \in \Mov(\M)$.  It suffices to 
 show that $\phi_{\mathbf{v}}$ is surjective for any $\mathbf{v}
\in \Mov(\M)$, and thus an isomorphism.  This shows that the
birational models of $\M$ have the form $V(J) \git_{\beta} H$ as
$\beta$ varies over the GIT chambers for \eqref{eqtn:newGIT}.  Since
the target of $\phi_{\mathbf{v}}$ is nonzero for all $\mathbf{v} \in
\Mov(\M)$, the source must be as well, which shows that $\mathbf{v}
\in E$, and thus $\Mov(\M) \subseteq E$.

Fix a monomial $m \in S$  of degree $l\mathbf{v}$ for some
$l>0$.  We will show that there is a
polynomial $f \in S$ with no term divisible by any $y_{\pi}$ with $m-f \in
I_{\M}$.  We proceed iteratively.  Set $f_0=m$, so $m-f_0 \in I_{\M}$
holds trivially.  Suppose that at some stage $i$ some monomial of
$f_i$ is divisible by $y_{\pi}y_{\pi'}$ with $\pi \neq \pi'$.  Let the
maximum number of distinct $y_{\pi}$ dividing any monomial in $f_i$ be
$k_i$, and choose one such monomial $\tilde{m}$.  We can then subtract
an appropriate multiple of one of the generators of $I_{\M}$ of the
form \eqref{item:case4} or \eqref{item:case5} to replace $\tilde{m}$.
This gives a polynomial $f_{i+1}$ with $m-f_{i+1} \in I_{\M}$.  Note
that the number of monomials in $f_{i+1}$ divisible by $k_i$ distinct
$y_{\pi}$ has decreased under this operation, so after a finite number
of iterations $k_i$ must also decrease.  As this cannot happen
indefinitely, after a finite number of iterations we have a polynomial
$f_k$ with $m-f_k \in I_{\M}$ and no term of $f_k$ divisible by more
than one distinct $y_{\pi}$.

Since $\mathbf{v} \in \Mov(\M)$, by Lemma~\ref{l:Qpihyperplane} the
vector $\mathbf{v}$ is contained in the half-space on the opposite
side of the hyperplane $\Delta_{\pi}$ from $Q_{\pi}$.  Thus every term
of $f_k$ divisible by some $y_{\pi}$ must also be divisible by some
$x_I$ with $\delta_I$ not on $\Delta_{\pi}$.  We can then subtract an
appropriate multiple of one of the the generators of $I_{\M}$ of the
form \eqref{item:case2} or \eqref{item:case3} to obtain a polynomial
$f_{k+1}$ which has one fewer monomial divisible the maximal occuring power of any $y_{\pi}$.
After a finite number of iterations we thus obtain a polynomial $f$
with $m-f \in I_{\M}$ and no term of $f$ divisible by any $y_{\pi}$.
This shows that $m$ is in the image of $\psi_{\mathbf{v}}$. Since the
monomials of degree $l\mathbf{v}$ as $l$ varies span $\oplus_{l \geq
  0} (S/I_{\M})_{l \mathbf{v}}$, this shows that $\psi_{\mathbf{v}}$
is surjective as required.
\end{proof}

\end{document}